\documentclass[12pt]{article}
\usepackage{amsmath,amsthm,amscd,amssymb}
\usepackage[mathscr]{eucal}
\usepackage{amssymb}
\usepackage{latexsym}

\newtheorem{Def}{Definition}[section]
\newtheorem{Thm}[Def]{Theorem}
\newtheorem{Prop}[Def]{Proposition}
\newtheorem{Rem}[Def]{Remark}

\newtheorem{Cor}[Def]{Corollary}

\numberwithin{equation}{section}
\newcommand{\Q}{\mathbb{Q}}
\newcommand{\R}{\mathbb{R}}
\newcommand{\C}{\mathbb{C}}
\newcommand{\Z}{\mathbb{Z}}

\newcommand{\hh}{\mathbb{H}}

\newcommand{\sym}{\mathrm{Sym}}

\newcommand{\trn}[1][1]{{}^t \hspace{-#1pt}}
\newcommand{\mat}[4]{\begin{pmatrix} #1 & #2 \\ #3 & #4 \end{pmatrix}}
\newcommand{\smat}[4]{\left(\begin{smallmatrix} #1 & #2 \\ #3 & #4 \end{smallmatrix}\right)}
\newcommand{\h}{\mathbb{H}}

\begin{document}

\title{Congruences for Siegel modular forms of nonquadratic nebentypus mod $p$}
\author{Siegfried B\"ocherer and Toshiyuki Kikuta}
\maketitle


\begin{abstract}
We prove that weights of two Siegel modular forms of nonquadratic nebentypus should satisfy some congruence relations if these modular forms are congruent to each other. 
Applying this result, we prove that there are no mod $p$ singular forms of nonquadratic nebentypus.
Here we consider the case where the Fourier coefficients of the modular forms are algebraic integers, 
and we emphasize that $p$ is a rational prime. 
Moreover, we construct some examples of mod $\frak{p}$ singular forms of nonquadratic nebentypus using the Eisenstein series studied by Takemori. 
\end{abstract}

\section{Introduction}
Serre \cite{Se} proved that the congruence $f_1\equiv f_2$ mod $p^r$ ($f_1\not \equiv 0$ mod $p$)  
implies $k_1\equiv k_2$ mod $(p-1)p^{r-1}$, where $p$ is an odd prime, $f_i$ ($i=1$, $2$) are modular forms for ${\rm SL}_2(\Z)$ of weights $k_i$. 
Later, this property was extended by Katz \cite{Katz} to the case of congruence subgroups of the type $\Gamma _1(N)$ with $p\nmid N$.
In Rasmussen \cite{Ra}, the case with algebraic integer Fourier coefficients is studied in detail, 
and the case of the group $\Gamma _0(p)\cap \Gamma _1(N)$ is also considered. 

All of the above results are for the case of elliptic modular forms. 
On the other hand, a generalization to the case of Siegel modular forms was given by Ichikawa \cite{Ichi} in an algebraic geometrical way and later given in an elementary way by the first author and Nagaoka \cite{Bo-Na}.
Here, the groups which they considered are of types $\Gamma (N)$ and $\Gamma _1(N)$ respectively ($p\nmid N$).

The first goal of this paper is to generalize the results mentioned above to the case 
of Siegel modular forms $F_i$ for a group $\Gamma _0(p^m)\cap \Gamma _1(N)$ ($m\ge 1$, $p\nmid N$) having a nebentypus characters $\chi _i$ mod $p^m$. 
The impetus for considering this congruence for weights in the case of nonquadratic nebentypus was to see if it could be applied to discuss the existence of a mod $p$ singular modular form of nonquadratic nebentypus. We explain this more precisely.

Over $\mathbb C$, singular Siegel modular forms are defined by the nonvanishing of all their Fourier coefficients of maximal rank; 
the properties of such modular forms were extensively studied by Freitag (see e.g. \cite{Fr}); 
in particular, he showed that there are no such modular
forms for $\Gamma_0(N)$ with nonquadratic character $\chi$ mod $N$.

In the works of the authors \cite{Bo-Ki,Bo-Ki2} on mod $p$ singular forms we only considered the case of quadratic characters. 
By applying our results on congruence with respect to weights as above, we find that there is no mod $p$ singular form
of nonquadratic nebentypus.
Here we consider the case where the Fourier coefficients of the modular forms are algebraic integers, 
and we emphasize that $p$ is rational prime. 
In fact, there exists a mod $\frak{p}$ singular modular form of nonqudratic nebentypus. 
We will show that it can be constructed from the Eisenstein series of nebentypus in some cases.

%%%%%%%%%%%%%%%%%%%%%%%%%%%%%%%%%%%%%%%%%%%%%%%
\section{Preliminaries}
\label{Sec:2}
\subsection{Siegel modular forms}
\label{sec:siegel-modular-forms}
Let $n$ be a positive integer and
$\hh_{n}$ the Siegel upper half space of degree $n$ defined as
\begin{equation*}
  \h_{n}:=\left\{X+ i Y \; | \;
    X,\ Y\in \sym_{n}(\R), \ Y>0
  \right\},
\end{equation*}
where $Y>0$ means that $Y$ is positive definite, and $\sym_{n}(R)$ is the set of symmetric matrices of size $n$ with components in $R$.

The Siegel modular group $\Gamma _n$ of degree $n$ is defined by 
\begin{equation*}
  \Gamma_{n} := \left\{\gamma \in \mathrm{GL}_{2n}(\Z)
    \; |\; \trn \gamma J_{n} \gamma = J_{n}
  \right\},
\end{equation*}
where $J_{n} = \smat{0_{n}}{-1_{n}}{1_{n}}{0_{n}}$ and $0_{n}$ (resp. $1_{n}$)
is the zero matrix (resp. the identify matrix) of size $n\times n$.

We define an action of $\Gamma _n$ on $\hh_{n}$ by
$\gamma Z := (AZ + B)(CZ + D)^{-1}$ for $Z \in \hh_{n}$, $\gamma = \left( \begin{smallmatrix} A & B \\ C & D \end{smallmatrix}\right) \in \Gamma _n$.
For a holomorphic function $F:\mathbb{H}_n\longrightarrow \mathbb{C}$ and a matrix $\gamma =\left( \begin{smallmatrix} A & B \\ C & D \end{smallmatrix}\right)\in \Gamma _n$,
we define a slash operator by
\[F|_k\; \gamma :=\det(CZ+D)^{-k}F(\gamma Z).\]

Let $N$ be a positive integer. 
In this paper, 
we deal with congruence subgroups of $\Gamma _n$ defined as 
\begin{align*}
&\Gamma _1(N):=\left\{ \begin{pmatrix}A & B \\ C & D \end{pmatrix}
\in \Gamma _n \: \Big| \: C\equiv 0_n \bmod{N},\ A \equiv D \equiv 1_n \bmod{N} \right\},\\
&\Gamma _0(N):=\left\{ \begin{pmatrix}A & B \\ C & D \end{pmatrix}\in \Gamma _n \: \Big| \: C\equiv 0_n \bmod{N} \right\}. 
\end{align*}
Here $A$, $B$, $C$, $D$ are $n \times n$ matrices.

Let $N$ be coprime to $p$. 
For a natural number $k$ and a Dirichlet character
$\chi : (\mathbb{Z}/p^m \mathbb{Z})^\times \rightarrow \mathbb{C}^\times $, the space
$M_k(\Gamma _0^{(n)}(p^m)\cap \Gamma _1^{(n)}(N), \chi )$
of Siegel modular forms of weight $k$ with
character $\chi$ (or nebentypus) consists of all of holomorphic
functions $F:\mathbb{H}_n\rightarrow \mathbb{C}$ satisfying
\begin{equation*}
%\label{trans}
(F|_{k}\: \gamma )(Z)=\chi (\det D)F(Z)\quad \text{for}\quad \gamma =\begin{pmatrix}A & B \\ C & D \end{pmatrix}\in \Gamma _0(p^m)\cap \Gamma _1(N).
\end{equation*}
If $n=1$, the usual condition in the cusps should be added.
When $\chi $ is a trivial character, we write simply $M^n_k(\Gamma _0(p^m)\cap \Gamma _1(N))$ for $M^n_k(\Gamma _0(p^m)\cap \Gamma _1(N),\chi )$.

Any $F \in M_k(\Gamma _0(p^m)\cap \Gamma _1(N),\chi )$ has a Fourier expansion of the form
\[
F(Z)=\sum_{0\leq T\in\Lambda_n}a_F(T)q^T,\quad q^T:=e^{2\pi i {\rm tr}(TZ)},
\quad Z\in\mathbb{H}_n,
\]
where
\[
\Lambda_n
:=\{ T=(t_{ij})\in {\rm Sym}_n(\mathbb{Q})\;|\; t_{ii},\;2t_{ij}\in\mathbb{Z}\}.
\]
We put $\Lambda _n^+:=\{T\in \Lambda _n \;|\;T>0 \}$. 

For a subring $R$ of $\mathbb{C}$, let $M^n_{k}(\Gamma _0(p^m)\cap \Gamma _1(N),\chi )(R)$ (resp. $M^n_{k}(\Gamma _0(p^m)\cap \Gamma _1(N))(R)$)
denote the $R$-module of all modular forms in $M^n_{k}(\Gamma _0(p^m)\cap \Gamma _1(N),\chi )$ (resp. $M^n_{k}(\Gamma _0(p^m)\cap \Gamma _1(N))$)
 whose Fourier coefficients are in $R$.

%%%%%%%%%%%%%%%%%%%%%%%%%%%%%%%%%%%%%%%%%%%%%%%%%%%%%%
\subsection{Congruences for modular forms}
Let $K$ be an algebraic number field, ${\mathcal O}_K$ the ring of integers in $K$.  
We define $\nu _\frak{p}(\gamma ):=e_\frak{p}$ for $\gamma \in K$, where $e_\frak{p}\in \Z$ is defined by 
$\gamma {\mathcal O}_K=\prod \frak{p}^{e_\frak{p}}$ (prime ideal factorization).
Let $F$ be a formal power series of the form
\begin{align*}
F=\sum _{T\in \Lambda _{n}}a_{F}(T)q^T
\end{align*}
with $a_{F}(T)\in K$ for all $T\in \Lambda _n$. 
We define 
\[\nu _{\frak{p}}(F):=\inf \{\nu _\frak{p}(a_F(T))\;|\; T\in \Lambda _n\}. \]
Let $F_i=\sum _{T\in \Lambda _{n}}a_{F_i}(T)q^T$ ($i=1$, $2$) be two formal power series as above. 
For an ideal $\frak{a}=\prod \frak{p}^{e_\frak{p}}$ in ${\mathcal O}_K$, suppose that $\nu _{\frak{p}}(F_i)\ge 0$ for all $\frak{p}$ with $\frak{p}\mid \frak{a}$. 
We write $F_1 \equiv F_2$ mod $\frak{a}$ if $\nu _\frak{p}(F_1-F_2)\ge e_{\frak{p}}$ for all prime ideals $\frak{p}$  with $\frak{p}\mid \frak{a}$.

Let $p$ be an odd rational prime and $\frak{p}$ a prime ideal in ${\mathcal O}_K$ such that $\frak{p}\mid p{\mathcal O}_K$. 
We denote by $e$ the ramification index $e=e(\frak{p}/p)$ of $\frak{p}$. 
Following Rasmussen \cite{Ra}, we take the Galois closure $L$ of $K$ 
and a prime ideal $\frak{P}$ in ${\mathcal O}_L$ with $\frak{P}\mid \frak{p}{\mathcal O}_L$. 
We put 
\[\beta (r):=\max\left\{ \Big \lceil \frac{r}{e}\Big \rceil-s-1,0\right\}, 
\] 
where $s:=\nu_p(\tilde{e})$ and $\tilde{e}:=e(\frak{P}/p)$ is the ramification index of $\frak{P}$. 
We remark that $\beta (r)=r-1$ for $K=\Q$. 
\begin{Thm}[Rasmussen \cite{Ra} Theorem 2.16]
\label{thm:Ra}
Let $N$ be a natural number with $p\nmid N$ and $F_i\in M_{k_i}^1(\Gamma _0(p)\cap \Gamma _1(N))({\mathcal O}_K)$. 
If $F_1\equiv F_2$ mod $\frak{p}^r$ and $F_1\not \equiv 0$ mod $\frak{p}$, then we have 
\[k_1\equiv k_2 \bmod{(p-1)p^{\beta (r)}}. \] 
\end{Thm}
\begin{Rem}
Rasmussen \cite{Ra} assumed $N\ge 3$. 
However the statement for $N\le 2$ follows from that for $N\ge 3$, since there is $M$ such that $\Gamma _1(N)\supset \Gamma _1(M)$ and $M\ge 3$.   
\end{Rem}

%%%%%%%%%%%%%%%%%%%%%%%%%%%%%%%%%%%%%%%%%%%%%%%%%%%%%%%%%%%
\subsection{An Eisenstein series congruent to a constant}
\label{Char}
Let $p$ be an odd rational prime and $\mu _{p-1}$ the group of $(p-1)$th root of unity in $\C^{\times}$. 
We consider the prime ideal factorization of $p$ in the ring $\Z[\mu _{p-1}]$ of integers in $\Q(\mu _{p-1})$.
We take a generator $\zeta _{p-1}$ of $\mu _{p-1}$. 
Let $\Phi(X)\in \Z[X]$ be the minimal polynomial of $\zeta _{p-1}$, namely $\Phi(X)$ is the cyclotomic polynomial having the root $\zeta_{p-1}$. 
Then $\Phi(X)$ can be decomposed in the form $\Phi (X)\equiv q_1(X)\cdots q_r(X)$ mod $p$, were $r=\varphi (p-1)$ and each $q_i(X)$ is a polynomial of degree $1$ satisfying $q_i(X)\not \equiv q_j(X)$ mod $p$. 
If we write $q_i(X)=X-d_i$ with $d_i\in \Z$, then we have $d_i\not \equiv d_j$ mod $p$ because of $q_i(X)\not \equiv q_j(X)$ mod $p$. 
Then $p$ is decomposed as a product of $r$ prime ideals $\frak{p}_i:=(\zeta _{p-1}-d_i,p)$, 
namely we have the perfect decomposition 
\[p\Z[\mu _{p-1}]=\frak{p}_1\cdots \frak{p}_r=(\zeta _{p-1}-d_1,p)\cdots (\zeta _{p-1}-d_r,p).\]
For this fact, we refer to Washington \cite{Wa}, p.15. 

We take a prime ideal $\frak{p}_i$ from above.  
We can define a character $\psi_i $ :$(\Z/p\Z)^\times \to \C^\times $ so that 
\[\psi_i (m)m\equiv 1 \bmod{\frak{p}_i}\]
 for all $m$ with $1\le m \le p-1$.  
Note that $\psi_i (-1)=-1$ and $\psi _i\neq \psi _j$ for all $i$, $j$ with $i\neq j$. 
We shall say that $\psi _i$ is the character corresponding $\frak{p}_i$.

%Let $\psi _i$ be the character corresponding a prime ideal $\frak{p}_i$ mentioned in previous subsection.   
Let $E_{1,\psi _i}$ be the Eisenstein series of weight $1$ for $\Gamma _0(p)$ with character $\psi _i$ of the form
\[E_{1,\psi_i }=1-\frac{2}{B_{1,\psi _i }}\sum _{n=1}^{\infty }\left( \sum _{d\mid n}\psi _i (d) \right)q^n,\]
where $B_{1,\psi _i}$ is the first generalized Bernoulli number with character $\psi _i$ described as
\[B_{1,\psi _i}=\frac{1}{p}\sum _{m=1}^{p-1}\psi_i(m)m. \] 
\begin{Thm}[Lang \cite{Lan} Theorem 1.2, p.250, Rasmussen \cite{Ra}, p.12]
\label{thm:Lang}
We have $\nu _{\frak{p}_i}(E_{1,\psi _i})\ge 0$ and  
\[E_{1,\psi _i } \equiv 1 \mod{\frak{p}_i}.\]
In other words, if we take $\gamma \in \Z[\mu _{p-1}]$ with $\gamma \not \in \frak{p}_i$ such that 
${\mathcal E}_{1,\psi_i}:=\gamma E_{1,\psi _i}\in M_1(\Gamma _0(p),\psi_i)(\Z[\mu _{p-1}])$,
then ${\mathcal E}_{1,\psi_i}\equiv \gamma \not \equiv 0$ mod $\frak{p}_i$ holds. 
\end{Thm}
\begin{Rem}
\label{rem:Eis}
\begin{enumerate}
\item 
The existence of $\gamma $ in the statement follows from the strong approximation theorem for the number fields. 
\item
As we will see later, we have
\[E_{1,\psi _i } \not \equiv 1 \mod{p\Z[\mu _{p-1}]}\]
 for any rational prime $p$ with $p\ge 5$. 
In other words, there exists $\frak{p}_j$ with $i\neq j$ such that 
\[E_{1,\psi _i } \not \equiv 1 \mod{\frak{p}_j},\]
and equivalently $pB_{1,\psi_i}\equiv 0$ mod $\frak{p}_j$. 
We do not know if $\nu_{\frak{p}_j}(E_{1,\psi _i })\ge 0$ to begin with. 
\item
Lang \cite{Lan} states that $E_{1,\psi _i } \equiv 1$ mod $p$ and 
Rasmussen \cite{Ra} states that $E_{1,\psi _i }$ has a Fourier coefficients in $\Z[\mu _{p-1}]$. 
Both statements seem to be inaccurate. 
\end{enumerate}
\end{Rem}

%%%%%%%%%%%%%%%%%%%%%%%%%%%%%%%%%%%%%%%%%%%%%%%%%%%%%%%%%%%
\section{Main results and their proofs}
\subsection{Statements of main results}
The first main result concerns congruences for the weights of modular forms of nonquadratic nebentypus. 
\begin{Thm}
\label{thm:M}
Let $p$ be an odd rational prime and $K$ a number field including $\Q(\mu _{p-1})$. 
Let $\frak{P}$ be a prime ideal in ${\mathcal O}_K$ with $\frak{P}\mid p{\mathcal O}_K$. 
For a prime ideal $\frak{p}$ in $\Z[\mu _{p-1}]$ such that $\frak{p}\mid p\Z[\mu _{p-1}]$ and $\frak{P}\mid \frak{p}{\mathcal O}_K$,  
we take the character $\psi $ corresponding $\frak{p}$ (see Subsection {\ref{Char}}). 
Let $\chi _i$ ($i=1$, $2$) be two Dirichlet characters mod $p^{m_i}$ with conductor $m_i'$. 
We put $m:=\max\{m_1',m_2'\}$ and take $\alpha _i$ with $0\le \alpha_i \le  p-2$ such that $\chi _i^{p^{m-1}} =\psi ^{\alpha _i}$.  
If $F_i\in M^n_{k_i}(\Gamma_0(p^{m_i})\cap \Gamma_1(N),\chi _i )(\mathcal{O}_K)$ satisfies that $F_1\equiv F_2$ mod $\frak{P}^r$ and $F_1\not \equiv 0$ mod $\frak{P}$, then we have 
\[k_1\cdot p^{m-1}-\alpha _1\cdot p^{r-1}\equiv k_2\cdot p^{m-1}-\alpha _2\cdot p^{r-1} \bmod{(p-1)p^{\beta (r)}}. \]
\end{Thm}
\begin{Rem}
The theorem above is essentially a statement about characters mod $p$. 
In fact, both characters $\chi _i^{p^{m-1}}$ ($i=1, 2$) are defined mod $p$, because $\sharp (\Z/p^{m'_i}\Z)^\times =p^{m'_i-1}(p-1)$ and $\chi_i ^{p^{m-1}}=(\chi_i^{p^{m'_i-1}})^{p^{m-m'_i}}$ ($i=1$, $2$). 
Therefore we can take $\alpha _i$ with $0\le \alpha_i \le  p-2$ such that $\chi _i^{p^{m-1}} =\psi ^{\alpha _i}$.  
\end{Rem}

As a consequence of the above results, it can be shown that the characters are somewhat restricted when we consider congruences mod rational prime $p$. 
\begin{Cor}
\label{cor:M}
\begin{enumerate}
\item
If $F_i\in M^n_{k_i}(\Gamma_0(p)\cap \Gamma_1(N),\chi _i )(\mathcal{O}_K)$ satisfies that $F_1\equiv F_2$ mod $p\mathcal{O}_K$ and $F_1\not \equiv 0$ mod $\frak{p}$ for some $\frak{p}$ with $\frak{p}\mid p{\mathcal O}_K$, then we have $\chi _1=\chi _2 \cdot (*/p)^t$ for some $t\in \Z$. Here $(*/p)$ is the unique nontrivial quadratic character mod $p$.  

In particular, if $F_1$ is of quadratic nebentypus, then $F_2$ is of quadratic nebentaypus. 
\item 
If there exist ${\mathcal E}_{k,\chi}\in M^n_{k}(\Gamma_0(p),\chi )(\mathcal{O}_K)$ and $\gamma \in \mathcal{O}_K$ such that ${\mathcal E}_{k,\chi }\equiv \gamma \not \equiv 0$ mod $p\mathcal{O}_K$, then $\chi $ is quadratic. 
\end{enumerate}
\end{Cor}

\begin{Rem}
\begin{enumerate}
\item
Actually, in order to prove (1), we can replace the condition $F_1\equiv F_2$ mod $p{\mathcal O}_K$ by a weaker condition $F_1\equiv F_2$ mod $\frak{p}\frak{p}^\rho $. Here $\rho $ is the complex conjugate in ${\rm Aut}(\C)$. See the proof in Subsection \ref{subsec:pr_prop}
\item
From the assertion (2), Remark \ref{rem:Eis} (2) follows.
\end{enumerate}
\end{Rem}

As mentioned in Introduction, the above results provide a mod $p$ analogue of Freitag's result that there are no singular forms (over $\C$) of nonquadratic nebentypus.
We just write down the simplest case. 

For a character $\chi $ mod $p$, we put $\Q(\chi ):=\Q(\chi (a)\:|\:a\in \Z)$. 
Let $\Z[\chi ]$ be the ring of integers in $\Q(\chi)$.  
\begin{Prop}
\label{prop:M}
Let $\chi $ be a character mod $p$ and $F\in M_k^n(\Gamma _0(p),\chi )(\Z[\chi])$. 
Assume that $F$ is strongly mod $p$ singular of $p$-rank $r$, i.e., $r\le n/2$, $a_F(T)\equiv 0$ mod $p\Z[\chi ]$ for all $T\in \Lambda _n$ with ${\rm rank}(T)<r$, and $a_F(T)\not \equiv 0$ mod $p\Z[\chi ]$ for some $T\in \Lambda _n$ with ${\rm rank}(T)=r$.  
Then $\chi $ is a quadratic character. 
\end{Prop} 
\begin{Rem}
We may as well consider mod $\frak{p}$ singular modular forms for prime ideal $\frak{p}$ in $\Z[\chi ]$.
Then our method (see Subsection \ref{subsec:pr_prop}) does not work when $\chi $ is nonquadratic character. 
Indeed, if we assume the existence of a modular form 
${\mathcal E}\in M_l^n(\Gamma _0(p),\chi )(\Z[\chi ])$ satisfying ${\mathcal E}\equiv \gamma \not \equiv 0$ mod $\frak{p}$ with $\gamma \in \Z[\chi ]$, we may switch between trivial nebentypus and nebentypus $\chi $ just by multiplication by ${\mathcal E}$. 
To the authors' knowledge, the existence of such ${\mathcal E}$ is not clear at all except degree $1$ and $2$ (see Theorem \ref{thm:Lang} and \cite{Ki} Theorem 1.2). 
For the higher degree case, in an unpublished note, Takemori constructed such ${\mathcal E}$ from Eisenstein series with character when $p$ is a regular prime (see also Remark \ref{rem:Ta}). 
\end{Rem}

%%%%%%%%%%%%%%%%%%%%%%%%%%%%%%%%%%%%%%%%%%%%%%%%%%%%%%%%%%%%%%
\subsection{Proof of Theorem \ref{thm:M}}
We first give two reductions:

\noindent 
{\bf Reduction 1: Decreasing the level to achieve $\boldsymbol {m_i=m_i'}$.}\\
As in the proof of Section 8.2 in \cite{Bo-Ki2} we change $F_i$ mod $\frak{P}^r$ to a form of level 
$\Gamma _0(p^{m_i'})\cap \Gamma _1(N)$ with same nebentypus character and with a weight $k_i'$
congruent to $k_i$ modulo a multiple of $(p-1)p^r$. 
We should keep in mind that $\beta (r)\le r$ always holds. \\

\noindent 
{\bf Reduction 2: Degree ${\boldsymbol 1}$ case is sufficient.}\\
We may apply the method of ``integral extract'' from \cite{Bo-Na}. 
We shortly describe the method (and the modifications necessary). 
As in \cite{Bo-Na} we may associate to $F_i$ ($i=1$, $2$) as above two elliptic modular forms
$f_i\in M_{k_i}(\Gamma _1(p^{m_i}NR^2))({\mathcal O}_K)$ satisfying 
$f_1\equiv f_2$ mod $\frak{P}^r$. 
In \cite{Bo-Na}, we ignored questions concerning nebentypus characters. 
The reduction is achieved by the observation that $f_i$ has indeed nebentypus character $\chi _i$ on 
$\Gamma _0(p)\cap \Gamma _1(NR^2)$.

We sketch the modifications necessary for the constriction of an ``integral extract'' which should keep track of nebentypus characters:
\begin{Prop}
For $F\in M_k^n(\Gamma _0(p^m)\cap \Gamma _1(N),\chi)(K)$ there exists, for any sufficiently large $R\in \mathbb{N}$ an elliptic modular form $f\in M^1_k(\Gamma _0(p^m)\cap \Gamma _1(NR^2),\chi)(K)$ such that 
\begin{itemize}
\item $f$ is an ``integral extract'' of $F$, i.e., the Fourier coefficients are finite sums of Fourier coefficients of $F$,
\item
$\nu_\frak{p}(F)=\nu_\frak{p}(f)$.
\end{itemize}
\end{Prop}
\begin{proof}
We use the construction from \cite{Bo-Na} and show that it satisfies the necessary transformation properties for $\Gamma_0(p^m)\cap \Gamma_1(NR^2)$. 
The properties for the smaller group $\Gamma _1(p^mNR^2)$ hold anyway. 
We recall from \cite{Bo-Na} that $f$ is obtained (for a suitable $T_0$) from 
\[F^{(R,T_0)}=\frac{1_n}{R^{\frac{n(n+1)}{2}}}\sum _J F|_k\; \mat{1_n}{\frac{J}{R}}{0}{1_n}\cdot e^{-2\pi i {\rm tr}(\frac{T_0\cdot J}{R})}, \]
where $J$ runs over all symmetric integral matrices mod $R$ of size $n$. 
Let $\smat{A}{B}{C}{D}\in \Gamma _0(p^m)\cap \Gamma _1^o(NR^2)$. 
Then, for any integral symmetric matrices $J$ of size $n$ we have 
\[\mat{1_n}{\frac{J}{R}}{0}{1_n}\mat{A}{B}{C}{D}=X\cdot \mat{1_n}{\frac{J}{R}}{0}{1_n}\]
with 
\[X=\mat{A+\frac{J}{R}}{-\frac{AJ}{R}-\frac{JCJ}{R^2}+B+\frac{JD}{R}}{C}{-\frac{CJ}{R}+D}.\]
Then $X\in \Gamma _0(p^m)\cap \Gamma _1(NR^2)$ provided that $C\equiv 0_n$ mod $R^2$; 
note that $A\equiv 1_n\equiv D$ mod $R$. 
The assertion follows in a standard way. 
\end{proof}
\begin{Prop}
Let $\chi _i$ be two Dirichlet characters mod $p^{m_i}$ and $F_i\in M_{k_i}^n(\Gamma _0(p^{m_i})\cap \Gamma _1(N),\chi_i)({\mathcal O}_K)$. 
Suppose that $F_1\equiv F_2$ mod $\frak{p}^r$ and $F_1\not \equiv 0$ mod $\frak{p}$ for a prime ideal $\frak{p}$ in ${\mathcal O}_K$. 
Then for sufficiently large $R$ with $p\nmid R$, there exist integral extracts $f_i\in M_{k_i}^1(\Gamma _0(p^{m_i})\cap \Gamma _1(N),\chi_i)({\mathcal O}_K)$ such that 
$f_1\equiv f_2$ mod $\frak{p}^r$ with $f_1\not \equiv 0$ mod $\frak{p}$. 
\end{Prop}

\begin{proof}[Proof of Theorem \ref{thm:M}]
It suffices to prove the case of degree $1$ by Reduction 2. 
We may assume that $m_i$ is conductors of $\chi _i$ because of Reduction 1. 

Suppose that $f_i\in M^1_{k_i}(\Gamma_0(p^{m})\cap \Gamma_1(N),\chi _i )(\mathcal{O}_K)$, $f_1\equiv f_2$ mod $\frak{P}^r$ and $f_1\not \equiv 0$ mod $\frak{P}$. 
Then we can find some $\gamma \in {\mathcal O}_K$ with $\gamma \not \in \frak{P}$ such that
\begin{align*}
&\gamma f_i^{p^{m-1}}\{(E_{1,\psi })^{p-1-\alpha _i}\}^{p^{r-1}}\in M_{k_i\cdot p^{m-1}+(p-1-\alpha _i) p^{r-1}}(\Gamma _0^{(n)}(p^{m})\cap \Gamma_1(N))({\mathcal O}_K)\\
&\text{and}\quad \gamma f_1^{p^{m-1}}\{(E_{1,\psi })^{p-1-\alpha _1}\}^{p^{r-1}}\equiv \gamma f_2^{p^{m-1}}\{(E_{1,\psi })^{p-1-\alpha _2}\}^{p^{r-1}} \bmod{\frak{P}^r}. 
\end{align*} 
Note here that $\gamma f_1^{p^{m-1}}\{(E_{1,\psi })^{p-1-\alpha _1}\}^{p^{r-1}}\not \equiv 0$ mod $\frak{P}$. 
It follows from Theorem \ref{thm:Ra} that  
\[k_1\cdot p^{m-1}-\alpha _1\cdot p^{r-1}\equiv k_2\cdot p^{m-1}-\alpha _2\cdot p^{r-1} \bmod{(p-1)p^{\beta (r)}}.\] 
This completes the proof of Theorem \ref{thm:M}.
\end{proof}

%%%%%%%%%%%%%%%%%%%%%%%%%%%%%%%%%%%%%%%%%%%%%%%%%%%%%%%%%%%%%%
\subsection{Proofs of Corollary \ref{cor:M} and Proposition \ref{prop:M}}
\label{subsec:pr_prop}
\begin{proof}[Proof of Corollary \ref{cor:M}]
(1) Let $\rho $ be the complex conjugate in ${\rm Aut}(\C)$. 
Taking sufficiently large field $K$, we may assume that $K^{\rho}=K$. 

For the prime ideal $\frak{p}$ with $F_1\not \equiv 0$ mod $\frak{p}$, 
we take the character $\psi $ corresponding $\frak{p}$ and $\alpha _i$ with $0\le \alpha _i \le p-2$ such that $\chi _i=\psi ^{\alpha _i}$.  
The condition $F_1\equiv F_2\not \equiv 0$ mod $\frak{p}$ implies $k_1-k_2\equiv \alpha _1-\alpha _2$ mod $p-1$, because of Theorem \ref{thm:M}. 
On the other hand, we have
\begin{align*}
F_i\cdot F_i^\rho \in &M^n_{2k_i}(\Gamma_0(p)\cap \Gamma_1(N))(\mathcal{O}_K),
\quad F_1\cdot F_1^\rho \equiv F_2\cdot F_2^\rho \bmod{p\mathcal{O}_K}\\
&~~~~~~~~~~~~~~~~~\text{and}\quad  F_1\cdot F_1^\rho \not \equiv 0 \bmod{p{\mathcal O}_K}. 
\end{align*}
Taking a prime ideal $\frak{p}'$ with $\frak{p}'\mid p{\mathcal O}_K$ such that $F_1\cdot F_1^\rho \not \equiv 0$ mod $\frak{p}'$, and applying Theorem \ref{thm:M} (or Theorem \ref{thm:Ra}) to $F_i\cdot F_i^\rho $, we have $2k_1-2k_2\equiv 0$ mod $p-1$. 
These two congruences imply that $\alpha _1-\alpha _2\equiv 0$ mod $\frac{p-1}{2}$. 
We put $\alpha _1=\alpha _2+\frac{p-1}{2}\cdot t$ with $t\in \Z$. 
This shows that 
\begin{align*}
\chi_1=\psi ^{\alpha _1}=\psi ^{\alpha _2+\frac{p-1}{2}\cdot t}=\chi_2\cdot \psi ^{\frac{p-1}{2}\cdot t}=\chi_2\cdot \left(*/p\right)^t. 
\end{align*}
We obtain the claim of (1). \\
(2) The claim follows immediately from (1).   
\end{proof}
\begin{proof}[Proof of Proposition \ref{prop:M}]
In the same way as in \cite{Bo-Ki2} Corollary 4.3 and Remark 4.4, we can find $S\in \Lambda _{r}^+$ and $h\in M^r_k(\Gamma _0(p),\chi )(\Z[\chi ])$ such that $\theta _S^r\equiv h\not \equiv 0$ mod $p\Z[\chi ]$. 
Here, for $S\in \Lambda _{m}^+$ ($m$ even), $\theta _S^n\in M_{m/2}^n(\Gamma _0(L),\chi_S)$ denotes the theta series 
\begin{align}
\label{eq:theta}
\theta _S^{n}(Z):=\sum _{X\in \Z^{m,n}}e^{2\pi i({\rm tr}(S[X]Z))}\quad (Z\in \hh_{n}), 
\end{align}
$L$ is the level of $S$, $\chi _S$ is the associated quadratic character of $S$, and $\Z^{m,n}$ is the set of  $m\times n$ matrices with integral components. 
By Corollary \ref{cor:M} (1), this can not happen, unless $\chi $ is quadratic. 
\end{proof}
%%%%%%%%%%%%%%%%%%%%%%%%%%%%%%%%%%%%%%%%%%%%%%%%%%%%%%%
\section{Eisenstein series of nonquadratic nebentypus}
\label{sec:4}
We take a prime ideal $\frak{p}_i$ in $\Z[\mu_{p-1}]$ and $q_i(X)=X-d_i\in \Z[X]$ from Subsection \ref{Char}.  
Let $\omega $ be the Teichm\"uller character on $\Z_p$. 
Then an embedding from $\Q(\mu _{p-1})$ to $\Q_p$ corresponding to $\frak{p}_i$
is determined by $\sigma (\zeta _{p-1})=\omega (d_i)$. 
Let $\chi :(\Z/p\Z)^\times \rightarrow \Q(\mu _{p-1})$ be a character mod $p$. 
Then we can find $\alpha $ with $0\le \alpha \le p-2$ such that $\chi ^\sigma =\omega ^\alpha $. 
We remark that $\chi =\psi _i^{-\alpha }$ holds for the character $\psi _i$ corresponding $\frak{p}_i$, 
since $\omega (d_i)\equiv d_i\equiv (\psi _i^{-1}(d_i))^\sigma $ mod $p$.

Suppose that $\chi $ is nonquadratic primitive character mod $p$. 
For a positive integer $k$ with $k>n+1$,  
let $E^n_{k}=\sum _Ta_{k}(T)q^T$ be the Eisenstein series of weight $k$ and degree $n$ for $\Gamma _n$, and $E^n_{k,\chi }=\sum _Ta_{k,\chi}(T)q^T$ the Eisenstein series of weight $k$ and degree $n$ for $\Gamma _0(p)$ with character $\chi $ which studied by Takemori \cite{Ta}. Here we assumed that $\chi(-1)=(-1)^k$ for $E^n_{k,\chi}$. 
In stead of giving the direct definition of $E^n_{k,\chi}$, we just write down its Fourier expansion below. 

Suppose that ${}^tUTU={\rm diag}(T',0_{n-r})$ with $T'\in \Lambda _{r}^+$ for some $U\in {\rm GL}_n(\Z)$. 
Recall from Katsurada \cite{Kat} and Takemori \cite{Ta} that   
\begin{align*}
&a_{k}(T)=2^{[(r+1)/2]}L(1-k,{\boldsymbol 1})^{-1}\left(\prod _{i=1}^{[r/2]}L(1+2i-2k,{\boldsymbol 1})^{-1}\right)\\
&~~~~~\times \prod _{q} F_q(T',q^{k-r-1})\times 
\begin{cases}
1 \quad \text{if}\ r\ \text{odd},\\
L(1+r/2-k,\chi _{T'})  \quad \text{if}\ r\ \text{even},
\end{cases}
\end{align*}
\begin{align}
\label{eq:FC}
&a_{k,\chi}(T) =2^{[(r+1)/2]}L^{(p)}(1-k,\chi )^{-1}\left(\prod _{i=1}^{[r/2]}L^{(p)}(1+2i-2k,\chi ^2)^{-1}\right) \\
&~~~~~\times \prod _{q\neq p} F_q(T',\chi (q)q^{k-r-1})\times 
\begin{cases}
1 \quad \text{if}\ r\ \text{odd},\\
L^{(p)}(1+r/2-k,\chi _{T'}\chi )  \quad \text{if}\ r\ \text{even}. 
\end{cases}
\end{align}
Here we put $L^{(p)}(s,\psi):=(1-\psi (p)p^{-s})L(s,\psi)$, ${\boldsymbol 1}$ is the trivial character mod $1$, 
and $F_q(T,X)\in \Z[X]$ is a polynomial with constant term $1$ calculated by Katsurada \cite{Kat}. 

Let $L_p(s,\psi )$ be the Kubota--Leopoldt $p$-adic $L$-function. 
Then $L_p(s,\psi )$ is a $p$-adic meromorphic (analytic if $\psi \neq {\boldsymbol 1}$) function on $\Z_p$ satisfying 
\begin{align}
\label{eq:p-adicL}
L_p(1-s,\psi )=-(1-\psi \omega ^{-s}(p)p^{s-1})
\frac{B_{s,\psi \omega ^{-s}}}{s}\quad (s\in \mathbb{Z}_{\ge 1}),  
\end{align}
and $L_p(s,{\boldsymbol 1})$ is analytic except for a pole at $s=1$.  
Here we regard as $\psi \omega ^{-s}={\boldsymbol 1}$ when $\psi =\omega ^s$.
On these facts, see Washington \cite{Wa}, p.57.

Then we have
\begin{align}
\label{eq:Lp}
L^{(p)}(1-s,\psi \omega ^{-s})=L_p(1-s,\psi)
\end{align}
for any $s\in \mathbb{Z}_{\ge 1}$.

%%%%%%%%%%%%%%%%%%%%%%%%%%%%%%%%%%%%%%%%%%%%%%%%%%%%%%%%%%%%%%%%%
\subsection{Mod $p$-power singular forms of nonquadratic nebentypus}
\begin{Thm}
\label{thm:sing}
Let $l$ be a natural number and $p$ a regular prime with $p>2l+1$.  
We take $\alpha \in \Z$ with $1\le \alpha \le p-2$ ($\alpha \not \equiv 0$ mod $\frac{p-1}{2}$) such that $\chi ^\sigma =\omega ^\alpha $. 
Let $\{l_{\delta, m}\}$ ($\delta =0$, $1$) be two sequences of positive integers such that $l+\frac{p-1}{2^\delta }$ is even and 
\[l_{\delta, m}\to \left(l,l-\alpha +\frac{p-1}{2^\delta }\right) \in {\boldsymbol X},\quad l_{\delta,m}\to \infty \quad (m\to \infty ). \] 
Then we have 
\[\lim _{m\to \infty }(E^n_{l_{\delta, m},\chi})^{\sigma }=\widetilde{E}^n_{(l,l+(p-1)/2^\delta )}. \]
Here $\widetilde{E}^n_{(l,l+(p-1)/2^\delta )}$ is the $p$-adic Eisenstein series studied in the authors \cite{Bo-Ki3}. 
\end{Thm}
\begin{proof}
We may assume that $l_{\delta ,m}=l+a(m)p^{b(m)}$ with $a(m)\equiv -\alpha +(p-1)/2^\delta $ mod $p-1$, $b(m)\to \infty $ ($m\to \infty$). 
Recall that   
\begin{align*}
&a_{l_{\delta ,m},\chi}(T)^\sigma =2^{[(r+1)/2]}L^{(p)}(1-l_{\delta ,m},\omega ^\alpha )^{-1}\left(\prod _{i=1}^{[r/2]}L^{(p)}(1+2i-2l_{\delta ,m},\omega ^{2\alpha })^{-1}\right) \\
&~~~~~\times \prod _{q\neq p} F_q(T',\omega ^{\alpha }(q)q^{l_{\delta ,m}-r-1})\times 
\begin{cases}
1 \quad \text{if}\ r\ \text{odd},\\
L^{(p)}(1+r/2-l_{\delta ,m},\chi _{T'}\omega ^\alpha )  \quad \text{if}\ r\ \text{even}. 
\end{cases}
\end{align*}

We consider the case where $s=l_{\delta ,m}$ and $\psi \omega ^{-s}=\omega ^\alpha $ in the formula (\ref{eq:Lp}). 
Then we have $\psi =\omega ^{\alpha +l_{\delta ,m}}=\omega ^{l+(p-1)/2^\delta }$ and therefore
\begin{align*}
L^{(p)}(1-l_{\delta ,m},\omega ^\alpha )=L_p(1-l_{\delta ,m},\omega ^{l+(p-1)/2^\delta }). 
\end{align*}
It follows that 
\begin{align*}
\lim _{m\to \infty }L_p(1-l_{\delta ,m},\omega ^{l+(p-1)/2^\delta }) =L_p(1-l,\omega ^{l+(p-1)/2^\delta })=L^{(p)}\left(1-l,(q/p)^\delta \right). 
\end{align*}
This shows 
\begin{align*}
\lim _{m\to \infty } L^{(p)}(1-l_{\delta ,m},\omega ^\alpha )=L^{(p)}\left(1-l,(q/p)^\delta \right). 
\end{align*}

Similarly, we have the following $p$-adic limits. 
\begin{align*}
&\lim_{m\to \infty }L^{(p)}(1+2i-2l_{\delta ,m},\omega ^{2\alpha })=\lim _{m\to \infty }L_p(1+2i-2l_{\delta ,m},\omega ^{2l-2i})\\
&~~~~~~~~~~~~~~ =L_p(1+2i-2l,\omega ^{2l-2i})\quad \text{for\ all\ }i\ \text{with}\ 1\le i\le [r/2],\\
&\lim_{m\to \infty }L^{(p)}(1+r/2-l_{\delta ,m},\chi _{T'}\omega ^\alpha )=\lim_{m\to \infty }L_p(1+r/2-l_{\delta ,m},\chi _{T'}\omega ^{l+(p-1)/2^\delta -r/2}), \\
&~~~~~~~~~~~~~~~~~~~~~~~~= L_p(1+r/2-l,\chi _{T'}\omega ^{l+(p-1)/2^\delta -r/2}) \quad \text{for\ }r\ \text{even}. 
\end{align*}

Here, due to the convergence of the factor $L^{(p)}(1+2i-2l_{\delta ,m},\omega ^{2\alpha })^{-1}$, the condition of regularity on $p$ is necessary. 
In fact, if $p$ is a regular prime, then $L_p(1+2i,\omega ^{-2i})\neq 0$ for all $i$ with $i\in \Z_{\ge 1}$
and hence 
\[\lim_{m\to \infty }L^{(p)}(1+2i-2l_{\delta ,m},\omega ^{2\alpha })=L_p(1+2i-2l,\omega ^{2l-2i})\neq 0. \] 
This means that $L^{(p)}(1+2i-2l_{\delta ,m},\omega ^{2\alpha })^{-1}$ converges. 

We consider the factor $F_q(T',\omega ^\alpha (q)q^{l_{\delta, m}-r-1})$. 
Since 
\begin{align*}
\omega ^\alpha (q)q^{l_{\delta ,m}-r-1}&\equiv \omega ^\alpha (q) q^{a(m) p^{b(m)}}q^{l-r-1}\\
&\equiv \omega ^\alpha (q) \omega ^{a(m)}(q)q^{l-r-1}\\
&\equiv \omega ^{(p-1)/2^\delta }(q)q^{l-r-1}\bmod{p^{b(m)+1}}
\end{align*}
and the constant term of $F_q(T',X)$ is $1$, 
we have 
\begin{align}
\label{eq:Fq}
&F_q(T',\omega ^\alpha (q)q^{l_{\delta, m}-r-1})\equiv F_q(T',\omega ^{(p-1)/2^\delta }(q)q^{l-r-1}) \bmod{p^{b(m)+1}}. 
\end{align}
This shows 
\begin{align*}
&\lim_{m\to \infty }F_q(T',\omega ^\alpha (q)q^{l_{\delta ,m}-r-1})= F_q(T',(q/p)^\delta q^{l-r-1}). 
\end{align*}
Summering these formulas, we have the following. 
\begin{align*}
&\lim _{m\to \infty} a_{l_{\delta ,m},\chi}(T)^\sigma =2^{[(r+1)/2]}L^{(p)}(1-l,(q/p)^\delta )^{-1}\left(\prod _{i=1}^{[r/2]}L_p(1+2i-2l,\omega ^{2l-2i})^{-1}\right) \\
&~~~~~\times \prod _{q\neq p} F_q(T',(q/p)^\delta q^{l-r-1})\times 
\begin{cases}
1 \quad \text{if}\ r\ \text{odd},\\
L_p(1+r/2-l,\chi _{T'}\omega ^{l+(p-1)/2^\delta -r/2})  \quad \text{if}\ r\ \text{even}. 
\end{cases}
\end{align*}

We need to check that this convergence is uniform in $T$.
The main issue is dependence on the quadratic character $\chi _{T'}$.   
If the conductor of $\chi _{T'}\omega ^{l+(p-1)/2^\delta -r/2}$ is divisible by a prime $q$ other than $p$, 
then there exists a formal power series $\Phi_{T'}(X)\in \Z_p[\![X]\!]$ such that 
$L_p(s,\chi _{T'}\omega ^{l+(p-1)/2^\delta -r/2})=\Phi_{T'}((1+p)^s-1)$ (cf. Washington \cite{Wa} Theorem 7.10). 
By this fact and (\ref{eq:Fq}), we see that the convergence is uniform in $T$.  

On the other hand, we take a sequence $\{k_{\delta ,m}\}$ of positive even integers such that $k_{\delta ,m}\to (l,l+(p-1)/2^\delta )\in {\boldsymbol X}$ ($m\to \infty $). 
We can confirm similarly that $L(1-k_{\delta ,m},{\boldsymbol 1})$, $L(1+2i-2k_{\delta ,m},{\boldsymbol 1})$, $F_q(T',q^{k_{\delta ,m}-r-1})$ and $L(1+r/2-k_{\delta ,m},\chi _{T'})$ have same $p$-adic limits respectively as above, and also that $\lim_{m\to \infty }F_p(T',p^{k_{\delta ,m}-r-1})=1$. 
Therefore we get 
\[\lim _{m\to \infty }(E^n_{l_{\delta, m},\chi})^{\sigma }=\lim _{m\to \infty}E^n_{k_{\delta, m}}=\widetilde{E}^n_{(l,l+(p-1)/2^\delta )}. \]
\end{proof}
\begin{Rem}
\begin{enumerate}
\item
We assumed that $p$ is regular. 
The condition we need actually is 
$L_p(1+2i-2l,\omega ^{2l-2i})\neq 0$ for all $l<j\le [n/2]$. 
Of course, if $p$ is regular, this condition holds.  
\item
Since $\widetilde{E}^n_{(l,l+(p-1)/2^\delta )}$ is mod $p$ singular of rank $2l$ by the results of authors \cite{Bo-Ki3}, we know that $\lim_{m\to \infty}(E_{l_{\delta, m},\chi })^\sigma $ is too, but we can also check directly by calculating the Fourier coefficients as follows.
Let $r>2l$. Then the factor $L^{(p)}(1+2i-2l_{\delta,m},\omega ^{2\alpha })^{-1}$ for $i=l$ appears in the Fourier coefficient of $a_{l_{\delta ,m},\chi}(T)^{\sigma}$. 
If we put $s=2l_m-2l$, $\psi =\omega ^{2l_m-2l+2\alpha }$ in (\ref{eq:Lp}), then we have $\psi \omega ^{-s}=\omega ^{2\alpha }$ and hence 
\[L^{(p)}(1+2i-2l_m,\omega ^{2\alpha })^{-1}=L_p(1+2l-2l_m,\omega ^{2l_m-2l+2\alpha })^{-1}.\] 
Since $\omega ^{2l_m-2l+2\alpha }={\boldsymbol 1}$ and $L_p(s,{\boldsymbol 1})$ has a pole at $s=1$, we have
\begin{align*}
\lim_{m\to \infty}L^{(p)}(1+2l-2l_m,\omega ^{2\alpha })^{-1}&=\lim_{m\to \infty}L_p(1+2l-2l_m,{\boldsymbol 1})^{-1}=0. 
\end{align*}
On the other hand, by the assumption, we have $L_p(1+2i-2l,\omega ^{2l-2i})\neq 0$ for all $1\le j\le [n/2]$. 
Therefore the other factors have $p$-adic limits in $\Q_p$, this shows that all Fourier coefficients for the higher degree vanish. 
\end{enumerate}
\end{Rem}

By the result for $\widetilde{E}^n_{(l,l+(p-1)/2^j)}$ in \cite{Bo-Ki3}, we obtain the following statement. 
\begin{Cor}
\label{cor:sing}
The following holds under the situation as Theorem \ref{thm:sing}.
\begin{enumerate}
\item
There exists some constant $\mu _\delta $ with $\mu _\delta \le 0$ such that  
\begin{align}
\label{eq:cong}
p^{-\mu _\delta }(E_{l_{\delta ,m},\chi}^{n})^\sigma \equiv 
\sum _{\substack{{\rm gen}(S)\\{\rm level}(S)\mid p\\ \chi _S=(*/p)^\delta }}
p^{-\mu _\delta }a_{\delta }({\rm gen}(S))\cdot (\Theta ^{n}_{{\rm gen}(S)})^0 \bmod{p^{c(m)}}, 
\end{align}
where $c(m)$ is a certain positive integer satisfying $c(m)\to \infty $ if $m\to \infty$. 
Here the summation goes over finitely many genera ${\rm gen}(S)$ of $S\in \Lambda_{2l}^+/{\rm GL}_{2l}(\Z)$ such that ${\rm level}(S)\mid p$ and $\chi _S=(*/p)^\delta $. 
\item
We have 
\[ \lim_{m\to \infty } (E^n_{l_{\delta, m},\chi})^{\sigma }=\sum _{\substack{{\rm gen}(S)\\ {\rm level}(S)\mid p\\ \chi _S=(*/p)^\delta }}a_\delta ({\rm gen}(S))\cdot (\Theta ^{n}_{{\rm gen}(S)})^0. \]
\end{enumerate}
Here, the notation is as follows. 
\begin{itemize} 
\item $(\Theta ^{n}_{{\rm gen}(S)})^0$ is the unnormalized genus theta series defined by 
\[(\Theta ^{n}_{{\rm gen}(S)})^0:=\sum _{S\in {\rm gen}(S)}\frac{1}{\epsilon (S)}\theta ^{n}_S, \] 
\item
$\theta ^{n}_S$ is as in (\ref{eq:theta}),
\item
$\epsilon (S)$ is the order of automorphism group of $S$, 
%\item
%$\chi _S$ is a Dirichlet character mod $N$ defined by 
%$\chi _S(d)={\rm sign} (d)^\frac{m}{2} \left( \frac{(-1)^\frac{m}{2}\det 2S}{|d|} \right)$,
\item 
$(*/p)^0$ is understood as the trivial character ${\boldsymbol 1}_p$ mod $p$, 
\item 
$a_\delta ({\rm gen}(S))\in \Q$ are certain constants coming from Fourier coefficients of the Eisenstein series of level $1$.  
\end{itemize}
\end{Cor}

\begin{Rem}
\label{rem:Ta}
\begin{enumerate}
\item
It is shown by Takemori in his unpublish note that there exists $E^n_{l,\chi}\in M^n_l(\Gamma _0(p),\chi )$ such that $(E^n_{l,\chi})^\sigma \equiv 1$ mod $p$ for some $l$, if $p$ is a regular prime.
For usual mod $p$ singular form $F\in M^n_k(\Gamma _0(p))$, by multiplying $(E^n_{l,\chi})^\sigma $, 
we can construct mod $p$ singular form $F(E^n_{l,\chi})^\sigma \in M^n_{k+l}(\Gamma _0(p),\chi ^\sigma )$ of nonquadratic nebentypus.  
\item
The embedding $\sigma $ gives an isomorphism of fields from $\Q(\mu _{p-1})$ to $\Q(\omega (d))$ 
($\subset \Q_p$) over $\Q$, where $d=d_i$ is as in the beginning of Section \ref{sec:4}.
Then ``mod $p^{c(m)}$'' in (\ref{eq:cong}) means ``mod $p^{c(m)}\Z_{(p)}[\omega (d)]$''.  
\item
Let $\frak{p}$ be a prime ideal in $\Z[\mu_{p-1}]$ such that the corresponding character is $\psi $ for which $(\psi ^{-1})^\sigma =\omega $.
Then (\ref{eq:cong}) corresponds to $\lambda _m\cdot p^{-\mu _\delta }E^n_{l_{\delta, m},\chi}$ being mod $\frak{p}^{c(m)}$ singular of $\frak{p}$-rank $2l$, where $\lambda _m$ is a rational integer with $v_p(\lambda _m)=0$ such that $\lambda _m\cdot p^{-\mu _\delta }E^n_{l_{\delta, m},\chi}$ has Fourier coefficients in $\Z[\mu _{p-1}]$. 
In fact, we can confirm this as follows. 
We define $\sigma ^{-1}:\Q(\omega (d))\to \Q(\mu _{p-1})$ by $\sigma ^{-1}(\omega (d)):=\zeta _{p-1}$.   
Then it is easy to see that $\sigma ^{-1}(p\Z[\omega (d)])=p\Z[\mu_{p-1}]\subset \frak{p}$. 
Hence we may apply $\sigma ^{-1}$ to the both sides of (\ref{eq:cong}).  
\end{enumerate}
\end{Rem}

%%%%%%%%%%%%%%%%%%%%%%%%%%%%
\section*{Acknowledgment}
This work was supported by JSPS KAKENHI Grant Number 22K03259.
%%%%%%%%%%%%%%%%%%%%%%%%%%%%%%%%%%

\section*{Data availability} 
During the work on this publication, no data sets were generated, used
or analyzed. Thus, there is no need for a link to a data repository.

%\bibliographystyle{amsplain}
%\bibliography{texref}

\providecommand{\bysame}{\leavevmode\hbox to3em{\hrulefill}\thinspace}
\providecommand{\MR}{\relax\ifhmode\unskip\space\fi MR }
\providecommand{\MRhref}[2]{%
  \href{http://www.ams.org/mathscinet-getitem?mr=#1}{#2}
}
\providecommand{\href}[2]{#2}

\begin{flushleft}
Siegfried B\"ocherer\\
Kunzenhof 4B \\
79177 Freiburg, Germany \\
Email: boecherer@t-online.de
\end{flushleft}

\begin{flushleft}
  Toshiyuki Kikuta\\
  Faculty of Information Engineering\\
  Department of Information and Systems Engineering\\
  Fukuoka Institute of Technology\\
  3-30-1 Wajiro-higashi, Higashi-ku, Fukuoka 811-0295, Japan\\
  E-mail: kikuta@fit.ac.jp
\end{flushleft}

\end{document}